\documentclass[10pt,english]{amsart}
\usepackage{times}

 \usepackage{epsfig}
\usepackage{amssymb}
\usepackage{amsmath}
\usepackage{amsthm}
\usepackage{graphics}
\usepackage{hyperref}

\usepackage{xcolor}

\newtheorem{Theorem}{Theorem}
\newtheorem{Corollary}[Theorem]{Corollary}

\newtheorem{thm}{Theorem}[section]

\newtheorem{lem}[thm]{Lemma}
\newtheorem{conj}[thm]{Conjecture}

\newtheorem{rem}[thm]{Remark}

\newcommand{\cA}{\mathcal{A}}
\newcommand{\cC}{\mathcal{C}}
\newcommand{\cD}{\mathcal{D}}

\newcommand{\cO}{\mathcal{O}}

\newcommand{\bH}{\mathbb{H}}
\newcommand{\bZ}{\mathbb{Z}}
\newcommand{\bR}{\mathbb{R}}
\newcommand{\bC}{\mathbb{C}}
\newcommand{\bA}{\mathbb{A}}
\newcommand{\bD}{\mathbb{D}}
\newcommand{\bN}{\mathbb{N}}
\newcommand{\bP}{\mathbb{P}}
\newcommand{\bL}{\mathbb{L}}
\newcommand{\bQ}{\mathbb{Q}}

\newcommand{\fS}{\mathfrak{S}}

\newcommand{\vv}{\vec{v}}
\newcommand{\ww}{\vec{w}}

\newcommand{\pa}{\partial}

\newcommand{\eps}{\varepsilon}
\newcommand{\la}{\lambda}

\newcommand{\Lau}{\bC(\!(t)\!)}
\newcommand{\Puis}{\bC\langle\!\langle t\rangle\!\rangle}

\DeclareMathOperator{\na}{na}

\DeclareMathOperator{\an}{an}

\DeclareMathOperator{\pr}{pr}
\DeclareMathOperator{\PGL}{PGL}
\DeclareMathOperator{\RPer}{RPer}
\DeclareMathOperator{\rat}{Rat}
\DeclareMathOperator{\brat}{\overline{Rat}}

\DeclareMathOperator{\modu}{\mathcal{M}}
\DeclareMathOperator{\bmodu}{\overline{\mathcal{M}}}

\begin{document}

\title{Blow-up of multipliers in meromorphic families of rational maps}

\author{Charles Favre}
\address{CNRS - Centre de Math\'ematiques Laurent Schwartz, 
	\'Ecole Polytechnique, 
	91128 Palaiseau Cedex, France}
\email{\href{charles.favre@polytechnique.edu}{charles.favre@polytechnique.edu}}

\date{\today}

\maketitle

\begin{abstract}
We study the blow-up of the multipliers of periodic cycles in one-parameter holomorphic degenerating families of 
rational maps of the Riemann sphere. 
\end{abstract}

\tableofcontents

\section*{Introduction}
The space $\rat_d$ of complex rational maps of degree $d\ge2$ 
is naturally a Zariski open subset of a projective variety $\brat_d$
which encodes all rational maps of degree $\le d$ and is isomorphic to
a projective space of dimension $2d+1$. 

The group of Möbius transformations $\PGL(2,\bC)$ acts on $\rat_d$ by conjugacy, and
Silverman proved that the quotient space $\modu_d= \rat_d/\PGL(2,\bC)$
is an affine space of dimension $2d-2$ having only quotient singularities, see~\cite{silverman}, and~\cite{BFN} for a thorough topological investigation 
of this space. It is also a rational variety
by a theorem of Levy~\cite{levy}.
Using GIT techniques, one can construct a natural projective compactification $\bmodu_d$ of $\modu_d$. 
Points at infinity in this compactification have no clear dynamical interpretation, and several attempts
have been done to find more adapted compactifications of $\modu_d$, see~\cite{demarco05,demarco07,schmitt}, and~\cite{fuji-tani,demarco-mcmullen}
for the analogous problem for the moduli space of polynomials.

\medskip

A meromorphic  family $\{f_t\}$ of rational maps of degree $d$ 
is by definition a holomorphic map $f\colon \bD \to \brat_d$ such that $f(\bD^*) \subset \rat_d$. 
It is said to be degenerating if moreover,
the induced class map $[f]\colon \bD\to \bmodu_d$ satisfies
$[f(0)]\notin\modu_d$.

In this note we investigate the behaviour of the multipliers of periodic orbits of $f_t$
as $t\to0$ in a general degenerating family. 
Our first theorem complements~\cite[Théorème~4.4]{FRLend}, and
can be stated as follows.
\begin{Theorem}\label{thm:main1}
Let $\{f_t\}$ be a degenerating family of rational maps of degree $d\ge2$.
Then we are in one of the following (exclusive) situations. 
\begin{enumerate}
\item
There exists a constant $C>0$ such that 
for any $|t|\le \frac12$ and for any $z\in\bP^1(\bC)$
such that $f_t^n(z)=z$, we have
\[
|df^n_t(z)|^{1/n} \le C
~.\]
\item
There exists $\la>0$ such that for any $\eps>0$ and for all $n\in\bN^*$, one can find a constant $C=C(\eps,n)>0$ for which 
\[
\frac1{d^n}
\# \left\{z\in\bP^1(\bC), \, f^n_t(z)=z, \text{ and } \max\{1,|df^n_t(z)|\}^{1/n}\ge C |t|^{-\la}\right\}
\ge 1-\eps
\]
for all $|t|\le \frac12$. 
\end{enumerate}
\end{Theorem}
In other words, we have the following strong dichotomy: either all multipliers are  bounded from above uniformly 
over the period $n$ and over the parameter $t$; 
or the multiplier of most periodic orbits blow-up when $t\to 0$.
In~\cite[Théorème~4.4]{FRLend} it was only proven that a positive proportion of multipliers blows up in the second case. 

In the case $f_t$ is a family of polynomials which is is stable over $\bD^*$, 
and its non-Archimedean limit has no recurrent critical points, then using results of Trucco~\cite{trucco}
we are able to obtain a uniform version of (2), see Theorem~\ref{thm:stable-uniform} below.

C. Gong~\cite{gong} has partially generalized the previous theorem to sequences of rational maps.

\medskip

We observe that in some special families, then only Case (2) can appear. 
In the case of polynomials, it follows from the work of DeMarco and McMullen~\cite{demarco-mcmullen}. 
For quadratic rational maps, it is due to Milnor~\cite{milnor-quadra} who proved that
a sequence in $\modu_2$ degenerates iff the modulus of the multiplier at a fixed point
diverges to $\infty$.
It is also the case for cubic rational maps by~\cite[Corollaire~4.5]{FRLend}.
In all these families, we prove that one can actually find a periodic point of period $\le3$
whose multiplier has a pole at $t=0$.

Let us say that a family of rational maps $\{f_t\}$ admits a fixed point
whose multiplier blows-up if there exists an irreducible component 
$C$ of the curve $\{(z,t),\ f_t(z)=z\}$ such that 
the multiplier function $(z,t)\in C\mapsto df_t(z)$ has at least one pole
over $t=0$.

\begin{Theorem}\label{thm:main2}
Let $\{f_t\}$ be any degenerating family of rational maps of degree $d\ge2$.
\begin{enumerate}
\item
If $\{f_t\}$ is a family of quadratic rational maps or of cubic polynomials, 
then there exists at least one fixed point
whose multiplier blows-up.
\item
If $\{f_t\}$ is a family of polynomials of any degree, then there exists at least one point of period $\le 2$
whose multiplier blows-up.
\item
If $\{f_t\}$ is a family of cubic rational maps, then there exists at least one point of period $\le 3$
whose multiplier blows-up.
\end{enumerate}
\end{Theorem}
Statement (1) is elementary, and 
the statement (2) was communicated to me by V.~Hu\-guin, see~\cite{huguin1,huguin2}. 
There are families of maps of any degree $d\ge4$ for which 
all multipliers of all periodic cycle stay bounded (Lattès examples in degrees that are square integers; 
or the McMullen family $z^d+ \frac{t}{z^d}$, see~\cite{mcmullen,DLU05,QWY12}). Also Luo has studied
hyperbolic components in $\modu_d$ for which there exists a degenerating sequence
whose mulitpliers are all bounded~\cite{luo2}. 

\medskip

Pick any rational map $f$ of degree $d\ge2$. 
By convention, a (resp. formal) $n$-cycle is a 
finite set of cardinality $n$ which is $f$-invariant
and over which the dynamics induced by $f$ is transitive (resp. 
is an  $m$ cycle with $m|n$ and the multiplier is a $n/m$-th primitive root of unity).
Let $\cD_n$ be the set of formal $n$-cycles of $f$ counted with multiplicity. 
Then $d_n:=\#\cD_n$
 only depends on $d$, see~\cite[Chapter 4.5]{silverman}. For small values of $n$ we have
$d_1=d+1$, $d_2=\frac12(d^2-d)$, and $d_3=\frac13(d^3-d)$. 
Define $\Lambda_n(f)$ as the collection of 
multipliers $\{df^n(z_i)\}$ where $z_i$ is an element in each cycle 
in $\cD_n$, and 
the collection $\{df^n(z_i)\}$ is viewed as an element of $\bC^{d_n}$ up to the 
action by the permutation group $\fS_{d_n}$.
An alternative way is to consider the symmetric functions of the multipliers of all cycles in $\cD_n$. 
In this way, we get a map  
$\Lambda_n\colon \rat_d\to \bA_\bC^{d_n}/\fS_{d_n}\simeq\bA^{d_n}_\bC$
which is algebraic, and descends to the moduli space
$\Lambda_n\colon \modu_d\to \bA^{d_n}_\bC$.

We can also consider the space of complex polynomials of degree $d\ge2$ modulo 
conjugacy by affine transformations. This is an affine space $\modu^{\mathrm{pol}}_d$ of dimension 
$d-1$ (see, e.g.,~\cite[Chapter~3]{favre-gauthier}). The number of formal $n$-cycles disjoint from 
$\infty$ is again given by $d_n$ if $n\ge2$, but
the number of fixed points different from $\infty$
is equal to $d'_1=d$. 
We construct as above maps
$\Lambda'_1\colon\modu^{\mathrm{pol}}_d\to\bA^{d}_\bC$, and 
$\Lambda_n\colon\modu^{\mathrm{pol}}_d\to\bA^{d_n}_\bC$ for any $n\ge2$. 

A result by Milnor implies that $\Lambda_2\colon\modu_2\to \bA^3_\bC$
is a proper embedding and its image is an affine plane, see~\cite{milnor-quadra}.
Also the map $\Lambda'_1\colon \modu^{\mathrm{pol}}_3\to\bA^3_\bC$ 
maps onto a plane $\Pi$, and $\Lambda'_1\colon   \modu^{\mathrm{pol}}_3\to\Pi$
is a bijective proper map, see~\cite{milnor-quadra,nishi-fuji}.

The preceding result has the following interesting consequence.
\begin{Corollary}\label{cor:embed}
\begin{enumerate}
\item 
the map $\Lambda\colon \modu^{\mathrm{pol}}_d\to\bA_\bC^d\times\bA_\bC^{\frac12(d^2-d)}$ 
given by $\Lambda=(\Lambda'_1,\Lambda_2)$
is a proper birational map;
\item
the map $\Lambda\colon\modu_3\to\bA^4_\bC\times\bA_\bC^3\times\bA_\bC^8$
given by $\Lambda=(\Lambda_1,\Lambda_2,\Lambda_3)$ is a proper birational map.
\end{enumerate}
\end{Corollary}
The fact that the map  is birational is proved in~\cite{HT}  in Case (1) and $d\le 5$, and 
in~\cite{gotou} in Case (2). Case (1) is proved in any degree in~\cite{huguin2}. 

Our main contribution in Case (1) is to give an independent proof of Huguin about the properness, and in Case (2) 
to show that $\Lambda$  is proper.
We dont know if the images of the maps above are normal. If it were the case, then we could conclude that $\Lambda$
is a proper embedding, but it is unlikely to be the case.
Note that since the multiplier maps are no longer proper on $\modu_d$
when $d\ge4$ so cannot  define proper embeddings. In a private communication, J.~Xie
has informed me that the multiplier maps are not even proper on their image.

\medskip
The proof of Theorem~A runs as follows. 
Following Kiwi~\cite{kiwi-cubic,kiwi}, we attach to our meromorphic family $f_t$ a natural
rational map $f_{\na}$ defined over the field of Laurent series $\Lau$. 
The latter field is complete for the $t$-adic norm $|\cdot|_t$, and
we analyze the dynamics of $f_{\na}$ over the Berkovich projective line
$\bP^{1,\an}_{\Lau}$. 

Recall first that one can define a natural ergodic measure $\mu_{f_{\na}}$
on $\bP^{1,\an}_{\Lau}$ (see, e.g.,~\cite{baker-rumely,theorie-ergo,benedetto-book}) that integrates 
the function $\log|df_{\na}|$. We may thus define the Lyapunov exponent of $f_{\na}$
by setting $\la(f_{\na}):=\int \log|df_{\na}| d\mu_{f_{\na}}$. 
Since the residual characteristic of $\Lau$ is $0$, it follows that
$\la(f_{\na})\ge0$, see~\cite{okuyama-approx-lyap,FRLend}.
When $\la(f_{\na})=0$, then we have a uniform bound on the multipliers by~\cite{FRLend}.


When $\la(f_{\na})>0$, one shall prove (Theorem~\ref{thm:main5}) that in a suitable sense almost every 
rigid periodic cycle for $f_{\na}$ are repelling with non-Archimedean multiplier $\ge \la(f_{\na})/2$, 
which implies Theorem~\ref{thm:main1}. 

\medskip

Let us now explain how we prove Theorem~\ref{thm:main2}. It is in fact a purely non-Archimedean statement
(see Theorems~\ref{thm:period2} and~\ref{thm:period3} for details). One needs to show that any polynomial (resp. any cubic rational map) over a metrized field of residual characteristic
$0$ has at least one repelling rigid point of period $\le 2$ (resp. $\le 3$)  except if it has potential good reduction. 
For polynomials, we use a counting argument. V. Huguin has given two other independent proofs of this fact, see~\cite{huguin2}. 
For cubic rational maps,  we initiate the investigation of the dynamics of a cubic rational map having at least one type II
fixed point of local degree $\ge 2$, in the vein of the fundamental work of Kiwi in the quadratic case~\cite{kiwi}. 
We hope that our work will lead to further developments in the classification of non-Archimedean cubic maps. 

\begin{conj}
For any $d\ge2$ there exists an integer $N(d)\ge1$ such that the following holds. 
Any rational map $f$ of degree $d$ defined over 
a field of residual characteristic $0$ 
with $\la(f)>0$ 
admits a rigid repelling periodic point of period $\le N(d)$. 
\end{conj}

Milnor proved $N(2)=1$, and we prove $N(3)=3$.

\medskip

The study of the behaviour of Lyapunov exponents in degenerating families was initiated by 
Faber and Demarco in dimension $1$, see~\cite{demarco-faber1,demarco-faber2}, 
and then further pursued in any dimension by the author~\cite{favre-degeneration}.
Powerful Berkovich techniques have been developed which even apply to 
sequences of maps~\cite{favre-gong}, or to families defined over a general Banach ring~\cite{poineau22}.

The behaviour of multipliers of periodic points in degenerating families of polynomials
was explored in~\cite{demarco-mcmullen}, and in the rational case lies at the root of the recent works of Luo~\cite{luo1,luo2}.

Finally the map associating to a rational map $f$ the multipliers at its periodic cycle
has been the focus of a lot of attention recently. 
A major advance was the proof by Z. Ji and J. Xie of the birationality of this map using a mixture of complex and non-Archimedean techniques~\cite{Ji-Xie,Ji-Xie23}.
We refer to these papers for more references on this map.

\medskip

\noindent {\bf Acknowledgements.}
I would like to thank both Chen Gong and Valentin Huguin for interesting discussions and  comments on a first version of this paper, that allowed me to correct
several inaccuracies.


\section{Degenerating families and non-Archimedean dynamics}

\subsection{Dynamics over $\Lau$}\label{sec:bas}
We refer to~\cite{benedetto-book,baker-rumely} for details. 
Recall that $\Lau$ endowed with the $t$-adic norm
\[\left|\sum_{n\in\bZ} a_nt^n\right|=e^{-\min\{n,\ a_n\neq0\}}\]
is a complete metrized non-Archimedean field. Its algebraic closure $\Puis$ is the field of Puiseux series
which carries a unique norm extending $|\cdot|$. It is not complete, but its completion $\bL$
is both algebraically closed and complete. 

The Berkovich projective line $\bP^{1,\an}_{\bL}$ is the union of the space of multiplicative semi-norms on 
$\bL[z]$ and a point $\infty$ endowed with the unique compact topology rendering all evaluation maps
$x \mapsto |P(x)|\in[0,+\infty]$ continuous. It is a compact $\bR$-tree. 

A type II point is a point determined by a norm of the form 
$P \mapsto \sup_{|z-z_0|\le r}|P(z)|$ where $z_0\in\bL$ and $r\in |\bL^*|=e^\bQ$. 
We denote it by $\zeta(z_0,r)$, and we write $x_g= \zeta(0,1)$ for the Gauss point (it is called like this since
the associated norm is the Gauss norm on $\bL[z]$).

Rigid points in  $\bP^{1,\an}_{\bL}$ are defined as 
the semi-norms having non trivial kernel (plus $\infty$) and are in bijection with $\bP^1(\bL)$. 
The complement of the set of all rigid points is denoted by $\bH_\bL$. 

It carries a canonical distance for which 
$\bH_\bL$ becomes a complete metric $\bR$-tree.

The Berkovich projective line $\bP^{1,\an}_{\Lau}$ is obtained as the quotient of  $\bP^{1,\an}_{\bL}$ by the action of 
the absolute Galois group $G$ of $\Lau$. It is again a compact $\bR$-tree, and rigid points are this time in bijection with
the set $\bP^1(\Puis)$ modulo the action of $G$. The $G$-action is isometric on $\bH_\bL$, and its quotient is 
the metric $\bR$-tree $\bH_{\Lau}$. 

\medskip

Any rational map $f$  of degree $d$ with coefficients in $\Lau$ defines a continuous finite $d$-to-$1$ map on  
$\bP^{1,\an}_{\Lau}$ (and on  $\bP^{1,\an}_{\bL}$).
This map preserves $\bH_{\Lau}$ (resp. $\bH_{\bL}$) and the set of rigid points.
 The group $\PGL(2,\bL)$  acts by isometries on $\bH_\bL$, and transitively both on rigid points and on type II points. 

Any rigid  point $z\in \bP^1(\bL)$ which is fixed by $f$ is in fact defined over $\Puis$, hence is rigid in $\bP^1_\Lau$. 
Let $\mu$ be the multiplier of $f$ at $z$. 
When $|\mu|>1$ (resp. $=1$, or $<1$) then we say that $z$ is a rigid repelling (resp. indifferent, or attracting) fixed point. 

Since $f$ is a finite map, one can define for any point $x\in\bP^{1,\an}_{\bL}$ the local degree $\deg_x(f)\in\{1,\cdots, d\}$.
On the set of rigid points, $\deg_x(f)=1$ if $x$ is not a critical point, and equal to $1$ plus the multiplicity as a critical point
if it is critical.

A rational map is said to have good reduction if the Gauss point is totally invariant. This means that $f(x_g)=x_g$ and 
$\deg_{x_g}(f)=d$. When it is the case, then for any representation $f=[P:Q]$ where
the maximum of the norm of the coefficients of $P$ and $Q$ is equal to $1$, then 
the complex rational map $\tilde{f}=[\tilde{P}:\tilde{Q}]$ has degree $d$ where $\tilde{P}$ and $\tilde{Q}$
are the reduction of $P$ and $Q$ respectively in $\bC$.

A rational map has potential good reduction it is conjugate by an element of $\PGL(2,\bL)$ to a rational map having good reduction. 
These maps are characterized as fixing a type II point $x$ such that $\deg_x(f)=d$.

\subsection{Lyapunov exponent}\label{sec:lyap}
Let $f$ be a rational map of degree $d\ge2$ defined over $\Lau$. 
Recall from~\cite{benedetto-book,baker-rumely,theorie-ergo} that for any point $x\in \bH_{\Lau}$
the sequence of probability measures $d^{-n} f^{n*}\delta_x$
converges to an ergodic measure $\mu_f$ referred to as the equilibrium measure. 
Its support is called the Julia set of $f$, which we denote by $J_f$. 
Then $f$ has potential good reduction iff 
$\mu_f$ has an atom iff $J_f$ is reduced to a type II point. 

\medskip

By \cite[Theorem~A]{FRLend} when $\mu_f$ charges a segment in $\bH_{\Lau}$, then $J_f$ is included in a compact 
segment in $\bH_{\Lau}$. More precisely, there exists a compact segment $I$ such that
$f^{-1}(I)$ is a union of $k\ge2$ sub-segments $I_j=[a_j,b_j]$ with 
$a_1<b_1\le a_2<b_2 \cdots \le a_k<b_k$ such that 
$f\colon I_j\to I$ is bijective and affine with integer slope $\pm d_j$ with $d_j\ge 2$ 
when the segments are parametrized by the distance $\bH_{\Lau}$.
Any such map is called affine Bernoulli, and has degree 
\begin{equation}\label{eq:bdd-bernoulli}
d=\sum_{j=1}^k d_j\ge 4~.
\end{equation}
When $f$ is not affine Bernoulli, and has not potential good reduction, then one can show that 
$\mu_f$ gives full mass to the set of rigid points (Case (1) of Theorem~C of op. cit.)

\medskip

It is a fact that $\mu_f = \delta_{x_g}+ \Delta g$ with $g$ continuous, 
hence  for any rational function $h\in \Lau(z)$, we have
$\log|h|\in L^1(\mu_f)$. 

Write $f= \frac{P}{Q}$ with $P,Q\in \Lau[z]$, $\max\{\deg(P), \deg(Q)\}=d$, and 
$P$ and $Q$ have no common factors. Set
\begin{equation}\label{eq:def-crit}
|df|(z) := \frac{|f'(z)|}{\max\{1,|f(z)|^2\}} \, \max\{1,|z|\}^2=
\frac{|P'Q-Q'P|}{\max\{|P|^2,|Q|^2\}} \,  \max\{1,|z|\}^2
~.\end{equation}
The Lyapunov exponent $\la(f)$ is defined as the quantity
\[
\la(f):= \int \log |df| d\mu_f
\]
This is a finite real number which is non-negative by~\cite{okuyama-approx-lyap,FRLend}.
The next  result is  ~\cite[Theorem~D]{FRLend}.
\begin{thm}\label{thm:main6}
Suppose $f$ is a rational map of degree $d\ge2$ defined over $\Lau$. 
Then the following statements are equivalent:
\begin{enumerate}
\item 
the  Lyapunov exponent $\la(f)$ vanishes;
\item
the rational map $f$ is affine Bernoulli;
\item
 $f$ has no repelling rigid periodic cycle.
\end{enumerate}
\end{thm}
The fact that (2) and (3) are equivalent was proved by Luo independently in~\cite{luo2}.
\subsection{Positive Lyapunov exponents and repelling cycles}
We shall need to estimate the number of repelling rigid periodic points of period $n$.
Denote by  $ \RPer(n)$ this set. 
It follows from~\cite[Theorem~B (3)]{FRLend} that 
$\#\RPer(n) \sim d^n$ when $\la(f)>0$.
Refining the proof of op.cit., we observe that we can obtain
a lower bound on the multipliers. 

\begin{thm}\label{thm:main5}
Suppose $f$ is a rational map of degree $d\ge2$ defined over $\Lau$, 
such that $\la(f)>0$. 

For each $\eps>0$, there exists  $A>0$ such that 
for all $n$ large enough
the number of rigid periodic points $z$ of period $n$ satisfying
\[
|(df^n)(z)| \ge A e^{n\la(f)/2}\] 
is at least $(1-\eps)d^n$.
\end{thm}

\begin{proof}
%
%
%
%
%
%
%

Denote by  $\hat{J} = \{\hat{x} = (x_i) \in J_f^\bZ,\,  f(x_i) = x_{i+1}\}$, the natural extension of $f$ over its Julia set, and $\pi\colon \hat{J}\to J_f$ the  projection map $\pi(\hat{x}) = x_0$. 
Let $\sigma$ be the right side shift so that $f\circ \pi = \pi \circ \sigma$, and let
$\hat{\mu}$ be the unique $\sigma$ invariant probability measure that projects onto $\mu_f$.

The next lemma is proved in~\cite{FRLend}.
\begin{lem} \label{l:branch2}
For any $L, \tau>0$, denote by
$E_{L, \tau}$ the set of points $\hat{x}\in \hat{J}$ such that for any integer $n\ge0$, $f^{n}$ admits an 
analytic inverse branch $f^{-n}_{\hat{x}}$
on $B(x_0, \tau)$  sending $x_0$ to $x_{-n}$ and satisfying  
$f^{-n}_{\hat{x}} (B(x_0, \tau)) \subset B(x_{-n}, L e^{-n\la(f)/2})$.

Then we have $\hat{\mu}(\cup_{L,\tau >0} E_{L,\tau})=1$.
\end{lem}

Fix $\epsilon >0$, and choose $L$ large enough and $\tau$ small enough such that $\hat{\mu} (E_{L,\tau}) \ge 1 - \epsilon$.
Pick any $r<\tau$, and $n_0$ large enough such that $ Le^{-n_0\la(f)/2} < r$.

For any $x \in \bP^{1,\an}(\bL)\cap J_f$, write $\hat{B}(x) = \pi^{-1} (B(x,r)) \subset \hat{J}$. 
The measure $\mu_f$ being mixing by~\cite[Proposition 3.5]{theorie-ergo}), the measure $\hat{\mu}$ is mixing too (\cite[p.241]{CFS}), and
we get
\[
\hat{\mu} \left( \sigma^{n} (\hat{B}(x)) \cap\hat{B}(x) \cap E_{L,\tau}\right) \to \hat{\mu} (\hat{B}(x))\, \hat{\mu} (\hat{B}(x)\cap E_{L,\tau})
= \hat{\mu} (\hat{B}(x)) \, \hat{\mu} (\hat{B}\cap E_{L,\tau})~.
\]

For all integer $n\ge n_0$ and for all $\hat{y} \in \sigma^{n} (\hat{B}(x)) \cap \hat{B}(x) \cap E_{L,\tau}$, 
Lemma~\ref{l:branch2} gives an analytic  inverse branch $f^{n}$ satisfying
\[
f^{-n}_{\hat{y}} (\hat{B}(x))
=
f^{-n}_{\hat{y}} (B(y_0, r)) \subset B(y_{-n}, L e^{-n\la(f)/2})\Subset B(x,r)~.\] 
Denote by $\mathcal{B}_n= \{ B_1, \cdots, B_s\}$ the set of open balls of the form $f^{-n}_{\hat{y}} (B(x, r))$ 
for some
$\hat{y} \in  \sigma^{n} (\hat{B}(x)) \cap \hat{B}(x) \cap E_{L,\tau}$.
The sets $\pi^{-1}(B_i)$ are disjoint, each one has mass equal to $d^{-n} \rho(B(x,r))$,  and they cover $ \sigma^{-n}(\sigma^{n} (\hat{B}(x)) \cap\hat{B}(x) \cap E_{L,\tau})$.
The cardinality of $\mathcal{B}_n$ hence satisfies
\[
\# \mathcal{B}_n
\times d^{-n} \rho(B(x,r))
\ge  
\hat{\mu} \left( \sigma^{n} (\hat{B}(x)) \cap\hat{B}(x) \cap E_{L,\tau}\right)
\] 
In each open ball $B_i = f^{-n}_{\hat{y}} (B(x, r)) \Subset B(x,r)$, Schwarz lemma
gives a periodic point $z:=\bigcap_{k \ge 0} f^{-nk}_{\hat{y}} (B(x, r))$ of period dividing $n$
with multiplier 
$|df^n(z)|\ge \frac{r}{L}e^{n\la(f)/2}$.

Let $\cA_n$ be the set of rigid periodic points of period dividing $n$ 
whose multiplier is larger than
$(r/L)e^{n\la(f)/2}$. We have obtained
\[
\liminf_n \frac1{d^n} \# (\cA_n \cap B(x,r)) \ge
\hat{\mu} (\hat{B}(x)\cap E_{L,\tau})~. \] 
Since the set of rigid points have full mass, and $\mu_f$ is regular, 
there exists a compact set $K\subset \bP^1(\bL)$
such that $\mu_f(K)\ge 1-\eps$. 
Pick any finite collection of disjoint 
balls $B(x_i,r)$ covering $K$.
Then $\{\hat{B}(x_i)\}$ are also disjoint and cover 
a set of $\hat{\mu}$-mass $\ge 1-\eps$ so that 
$\hat{\mu} (\cup_i\hat{B}(x_i)\cap E_{L,\tau})\ge 1-2\eps$.
We conclude that 
\[
\liminf_n \frac1{d^n} \# \cA_n \ge1-2\eps~, \] 
which implies the result with $A= r/L$.
\end{proof}

\begin{rem}\label{rem:also-complex}
The same argument works over any algebraically closed metrized field of residual characteristic $0$. 
It is also valid over $\bC$. In fact the presentation above and Lemma~\ref{l:branch2} follow closely~\cite[\S 3]{briend:Lyap}. 
\end{rem}

\subsection{Meromorphic families of complex rational maps}
Recall from the introduction our definition of degenerating family of rational maps of degree $d\ge2$. 
Let $\cO(\bD)$ be the ring of all holomomorphic maps on $\bD$, and 
$\cO(\bD)[t^{-1}]$ be the ring of all meromorphic maps on $\bD$ with at most one pole at $0$.

A meromorphic family of rational maps parameterized by the unit disk $\bD$
is a family $f_t= \frac{P_t}{Q_t}$ where $P_t$ and $Q_t$
are polynomials with coefficients in $\cO(\bD)[t^{-1}]$, 
and for each $t\neq 0$, $\max\{\deg(P_t),\deg(Q_t)\}=d$ and $P_t^{-1}(0)\cap Q_t^{-1}(0)=\emptyset$. 
Multiplying by a suitable power of $t$, we may (and shall) assume that 
$P_t, Q_t$ have coefficients in $\cO(\bD)$. 

We say that a meromorphic family $(f_t)_{t\in\bD}$ is degenerating when the induced class map 
$t \mapsto [f_t]\in\modu_d$ is  unbounded. 

Observe that $\cO(\bD)\subset \cO(\bD)[t^{-1}]$ are both subrings of $\Lau$, 
so that any family as above gives rise naturally of a rational map over $\Lau$
which we denote by $f_{\na}$. It has degree $d$
because $\deg(f_t)=d$ for all $t\neq0$. 

The dynamical properties of $f_{\na}$ acting on the projective Berkovich line
$\bP^{1,\an}_{\Lau}$ are very much intertwined with the one of $f_t$ on the Riemann sphere 
at least when $t\to0$, as the next two statements show.
\begin{thm}\label{thm:pot-good}
A meromorphic family $(f_t)_{t\in\bD}$ of degree $d\ge2$ is degenerating iff
$f_{\na}$ has not potential good reduction. 
\end{thm}
A sequential version of this result is given in~\cite[Theorem~3.17]{favre-gong}.
We give a quick proof of this result for the convenience of the reader. 

Observe that this result and the fact that $\modu_d$ is affine implies the characterization of isotrivial  
families over function fields, a theorem due to Baker~\cite{baker}, see~\cite{PST} for the higher dimensional case.
\begin{proof}
If $f_{\na}$ has potential good reduction, then we may suppose after base change and conjugating by a meromorphic family
of Möbius transformations defined over a neighborhood of $0$ that the Gauss point is totally invariant. 
This implies that $f_0$ is a rational map of degree $d$, see~\cite[Lemma~3.2]{kiwi-rescaling}, hence the map is not degenerating. 

Suppose conversely that $f_t$ is not degenerating. After base change and conjugacy by a meromorphic family
of Möbius transformation we may suppose that $\deg(f_0)= d$. Again by~\cite[Lemma~3.2]{kiwi-rescaling}, the induced map by $f$ on the space of directions at $x_g$ has degree $d$ and this 
implies $f_{\na}$ has good reduction. 
\end{proof}
For any $t\in\bD^*$, we define the equilibirum measure $\mu_t$ as the limit of $d^{-n}f^{n*}_t\omega$
where $\omega$ is the Fubini-Study form on the Riemann sphere. This is a probability measure which is mixing, 
and integrates all functions locally of the form $\log|h|$ for $h$ analytic so that the Lyapunov exponent
$\la(f_t):= \int \log|df_t| d\mu_t$ is well-defined (where $|df_t|$ is defined by~\eqref{eq:def-crit}). 
Margulis-Ruelle inequalities imply $\la(f_t)\ge\frac12\log d$.

We have:
\begin{thm}\label{thm:lyap}
Let $(f_t)_{t\in\bD}$ be any degenerating family of rational map of degree $d\ge2$.
Then $t\mapsto\la(f_t)$ is subharmonic, and we have
\[
\frac{\la(f_t)}{\log|t|^{-1}} \to \la(f_{\na})
~.\]
\end{thm}
This result is due to Demarco~\cite{demarco03,demarco16}. A quantitative version is given in~\cite{GOV1} and will be used below in \S\ref{sec:blw-GOV}. 
A higher dimensional version is proved in~\cite{favre-degeneration}.

\subsection{Multipliers of periodic cycles}
Let $(f_t)_{t\in\bD}$ be any meromorphic family of rational maps of degree $d\ge2$. 
Pick any base point $t_0\in \bD^*$.

For any $n\in\bN^*$, the closure of the set 
\[
C_n:=
\{(z,t)\in\bP^1_\bC\times \bD^*, \ f_t^n(z)=z\}
\]
is an analytic subset of $\bP^1_\bC\times \bD$
for which the second projection 
$\pr_2\colon C_n\to \bD$ is a finite ramified cover of 
degree at most $d_n:=\sum_{m|n} \mu(\frac{n}{m})(d^m+1)$, where $\mu$
is the Möbius function.
The multiplier function 
$\mu_n\colon C_n\to \bC$ is the analytic function defined by
\[\mu_n(z,t)= (df_t^n)(z)~.\]
We also define the function $\la_n\colon C_n\to  \bR_+$
by setting $\la_n(z,t)= \frac1n \log^+|\mu_n(z,t)|$. 
This is a subharmonic function on (the normalization of) $C_n$.

Let now $z_0\in\bP^1_\bC$
be any periodic point of $f_{t_0}$ of exact period $n$, and consider  $C_{z_0}$ the irreducible component of $C_n$ containing
$(z_0,t_0)$.  Let $\tilde{C}_{z_0}$ be the normalization of $C_{z_0}$, and
denote by $\tilde{\pr}_2\colon \tilde{C}_{z_0}\to\bD$ 
the composite of $\pr_2$ with the normalization map. 
Each point $p\in \tilde{\pr}_2^{-1}(0)\subset \tilde{C}_{z_0}$ defines a branch of $C_{z_0}$ 
at a point $(z_p,0)$, and is thus determined by a Puiseux series
$\hat{z}_p\in\Puis$. This Puiseux series is uniquely determined
only up to the Galois action of the absolute Galois group of $\Lau$. 
Set
\[
\la^+_{\na}(z_0):= \frac1n \max_{p\in\tilde{\pr}_2^{-1}(0)} \log^+|(df_{\na}^n)(\hat{z}_p)|~,
\]
and
\[
\la^+_{z_0}(t) = \max\{\la_n(z,t),\, (z,t)\in C_{z_0}\}
\]
\begin{lem}\label{lem:multi-na}
For any periodic point $z_0$ of period $n$ for $f_{t_0}$, the function
\[
\la^+_{z_0}(t) - \la^+_{\na}(z_0) \log|t|^{-1}
\]
is a continuous subharmonic function on $\bD^*$ which is bounded locally near $0$
(hence extends through the origin). 
\end{lem}
\begin{proof}
We argue locally near a parameter $t_0$. 
Over a small neighborhood $U$ of $t_0$, then 
$\tilde{\pr}^{-1}_2(U) = \sqcup_i \tilde{U}_i\subset \tilde{C}_{z_0}$
(and ${\pr}^{-1}_2(U) = \sqcup_i {U}_i\subset {C}_{z_0}$). 
Reducing $U$ if necessary, we can assume that
the restriction of the normalization map $\tilde{U}_i \to U_i$
is an isomorphism on $\tilde{U}^*_i:=\tilde{U}_i\setminus \pr_2^{-1}(t_0)$; 
that $\tilde{U}^*_i \to U^*:=U\setminus\{t_0\}$ is an unramified cover; 
and $0\notin U$ if $t_0\neq0$. 

Suppose first that $t_0\neq 0$. For each $t\in U^*$, 
define 
\[\la_i(t):=\max_{\tau\in\pr_2^{-1}(t) \cap U_i} \frac1n \log^+|\mu_n(\tau)|~.\]
The function $\log^+|\mu_n|$ is continuous and subharmonic
on $C_{z_0}$. 
Since the maps $U_i\to U$ are ramified covers, 
$\la_i$ is subharmonic on $U^*$ and continuous, hence
extends to $U$ as a continuous subharmonic function. 
This implies $\la^+=\max_i\la_i$ to be subharmonic and continuous on $U$ as required.

In the case $t_0=0$, then each $U_i$ corresponds to a branch $p_i$ of $C_{z_0}$.
Choose a parameter $\tau$ on $\tilde{U}_i$ such that
the normalization map is of the form
$\tau\mapsto (z(\tau), \tau^{q_i})$ for some $q_i\in\bN^*$,
The Puiseux series $\hat{z}_i= z(\tau^{1/q_i})$ determines the branch associated with $p_i$, and
on $\tilde{U}_i$, we can write
\[
\frac1n \log^+|\mu_n(\tau)|= 
 \frac{q_i}n \log^+|(df_{\na}^n)(\hat{z}_i)|\times
 \log|\tau|^{-1} + O(1)
 \]
Pushing this to $U$, we get 
the result.
\end{proof}

Suppose now that $\hat{z}\in\bP^1(\Puis)$ satisfies $f^n_{\na}(\hat{z})=\hat{z}$.  
We may suppose that $\hat{z}$ is determined by a formal Puiseux series vanishing at $0$, so that 
$\hat{z}(t)= \sum_{j\ge 1} a_j t^{j/q}$ with $q\in\bN^*$ and $a_j\in\bC$. 
The formal power series $\hat{z}(t^q)$ then satisfies a polynomial equation of degree $d+1$ whose coefficients lie
in the ring of holomorphic functions on the disk, which implies $t\mapsto\hat{z}(t^q)$ to be also holomorphic locally at $0$ by Artin approximation theorem~\cite{artin}.

We may thus attach to $\hat{z}$ the analytic curve 
\[
C_{\hat{z}}:=\{(\hat{z}(t^q),t^q), |t|<\eps\}
\]
for some $\eps>0$,  which is an open subset of $C_n$.

\section{Degenerating families with positive non-Archimedean Lyapunov exponent}

\subsection{Proof of Theorem~\ref{thm:main1}}
Suppose first that $\la(f_{\na})=0$. Then by~\cite[Théorème~4.4]{FRLend} we conclude to the existence of a constant $C>0$
such that  for any $0<|t|\le 1/2$, and any periodic point $z$ of period $n$ for $f_t$
then $|df_t^n(z)|^{1/n}\le C$. This proves that we are in Case (1). 

Suppose now that $\la(f_{\na})>0$, and pick any $\eps>0$. 
Let $A_1>0$ be the constant given by  Theorem~\ref{thm:main5}.
For each $n\in\bN$, denote by  $F_n$ the set of Puiseux series $\hat{z}$
such that $f_{\na}^n(\hat{z})=\hat{z}$ and $|df^n_{\na}(\hat{z})|\le A_1 e^{n\la(f_{\na})/2}$, so that
 $\# F_n \le \eps d^n$. 
 Pick  $0<\la<\la(f_{\na})/2$.

Choose any $n$, and let $r_n$ be a sufficiently small positive real number such that 
the second projection $\pr_2\colon C_n\to\bD$ is unramified over $\bD^*(r_n)$.
Let $\{C_i\}_{i\in I}$ be the irreducible components of  $\pr_2^{-1}(\bD^*(r_n))$. 
For each $i\in I$,  the projection $\pr_2\colon C_i\to \bD(r_n)$ is ramified only over $0$
hence $C_i$ is a punctured disk. If $q_i= C_i\cdot (t=0)$, then 
the curve $C_i$ can be parametrized by $q_i$ distinct Puiseux series
$\hat{z}^i_j$ such that $df_t^n(\hat{z}^i_j)$ are in the same orbit under the action of
the absolute Galois group of $\Lau$ hence have the same norm. 

Let us say that an index $i$ is good  if $|df_{\na}^n(\hat{z}^i_j)|\ge A_1e^{n\la(f_{\na})/2}$ for some (hence all) $j$. 
It follows that 
$\sum_{i \text{ good }} C_i\cdot (t=0)\ge (1-\eps) d^n$. 

Recall that $\la_n= \frac 1n \log^+|df_t^n(z,t)|$. 
By Lemma~\ref{lem:multi-na} applied to the family restricted over the disk of radius $r_n$, we get the existence of $A'_i>0$ such that
\[
\la_n \ge \left(\frac{\la(f_{\na})}2+\frac{\log A_1}n\right)\log|t|^{-1}  - A'_i
\]
on $C_i$ if $i$ is good. 
Reducing again $r_n$, we get
\[
\frac1{d^n}
\# \left\{z\in\bP^1(\bC), \, f^n_t(z)=z, \text{ and } \max\{1,|(f^n_t)'(z)|^{1/n}\}\ge A_+ |t|^{-\la}\right\}
\ge 1-\eps
\]
for all $|t|\le r_n$. Reducing $A_+$ if necessary, we may suppose this is true over $|t|\le1/2$, and the proof is complete.

\subsection{Stable families of polynomials}
A critical point of a family $f_t$ is a holomorphic map $t\mapsto c(t)$
such that $df_t(c(t))=0$. 
By making a suitable base change we can always assume that
we have $2d-2$ critical points for $f_t$ (at least on any disk $|t|<1-\eps$).

Recall from~\cite{berteloot-bifurcation} that a holomorphic family of rational maps $f_t$ is said to be stable
iff the family $\{f_t^n(c(t))\}_{n\in\bN}$ is normal for any critical point $c$.
This property is equivalent to say that for any irreducible component $C$ of $C_n = \{(z,t),\ f^n_t(z)=z\}$ 
such that $|df^n_{t_0}(z)|>1$ for some $(z,t_0)\in C$, then $|df^n_{t}(z)|>1$ for all $(z,t)\in C$. 

It follows that when the family is stable, then we can decompose $C_n$ into irreducible components
$C^+_i$ and $C^0_j$ such that $|df^n_{t}(z(t))|>1$ for all $(z,t)\in\cup_iC^+_i$, and
$|df^n_{t}(z(t))|\le 1$ for all $(z,t)\in\cup_jC^0_j$. 

\begin{thm} \label{thm:stable-uniform}
Let $\{f_t\}$ be any degenerating family of polynomial maps of degree $d\ge2$
such that $\la(f_{\na})>0$. 
Suppose that  the family $\{f_t\}$ is stable over $\bD^*$, and that
$f_{\na}$ has no recurrent critical points in $J(f_{\na})$. 

Pick any $0<\la<\la(f_{\na})/2$.
Then for any $\eps>0$ one can find a constant $C>1$ such that 
for all $n\in\bN^*$ large enough, we have
\[
\frac1{d^n}
\# \left\{z\in\bP^1(\bC), \, f^n_t(z)=z, \text{ and } |(f^n_t)'(z)|^{1/n}\ge C |t|^{-\la}\right\}
\ge 1-\eps
\]
all $|t|\le 1/2$.  
\end{thm}

\begin{rem}
This result applies for instance on any branch of a special curve, i.e.,
on an algebraic curve $C$ in the moduli space containing infinitely many 
post-critical polynomials. Any branch which is sufficiently tangent to a branch of a special curve
will also satisfy the assumptions of the theorem.

We refer to~\cite{favre-gauthier} for a classification of special curves. 
\end{rem}

\begin{proof}
Choose any $\eps>0$,  and let $A$ be the constant given by Theorem~\ref{thm:main5}.

Pick any integer $n\in\bN^*$. 
Consider the set ${\mathfrak C}$ of all irreducible components 
$C_i^+$ of $C_n$. Since the differential $df^n_t(z,t)$ has norm $>1$, 
it cannot be equal to $1$ hence by the implicit function theorem
$C^+_i$ is smooth at any point over $\bD^*$, hence
$\pr_2\colon C^{+,*}_i\to \bD^*$ is a finite ramified cover
where $C^{+,*}_i= C_i^+\cap \pr_2^{-1}(\bD^*)$.
In particular, $C^+_i$ is connected.

We reduce the set ${\mathfrak C}$ so that 
for any (equivalently for all) Puiseux series $\hat{z}$ defining $C^+_i$, 
then $|df_{\na}^n(\hat{z})|\ge A e^{n\la(f_{\na})/2}$. 
Then Theorem~\ref{thm:main5} implies that 
\[\sum_i (C^+_i\cdot (t=0))\ge (1-\eps) d^n.\] 
Observe that any Puiseux series defining a component $C_i^+$ lies in the Julia set of $f_{\na}$.
Also $f_{\na}$ has no recurrent critical points, hence by Trucco~\cite[Theorem~F]{trucco} there exists 
an integer $m\in\bN^*$ such that
$(C_i^+\cdot(t=0))\le m$ for all $i$ (the point is that the constant $m$ is uniform in $n$). 
In other words, the map $\pr_2\colon C^+_i\to \bD^*$ is a finite cover of degree at most $m$
for all $i$. 

Now consider the family $g_\tau = f_{\frac12 e^{2i\pi \tau}}$ for $\tau\in[0,1]$. 
The total number of continuous functions $\tau \mapsto w(\tau)$ such that 
$g^n_\tau(w(\tau))=w(\tau)$ is equal to $d^n+O(1)$. By Theorem~\ref{thm:uniform-rep} below, 
we can find a constant $B>1$ such that 
at least $(1-\eps) d^n$ such functions  satisfy
$|(dg^n_\tau)(w(\tau))|^{1/n}\ge B$ for all $\tau$. 
It follows that removing at most $m\eps d^n$ curves from $\mathfrak{C}$, we can suppose that 
for each $C^+_i$, we have 
\[|(df^n_t)(z(t))|^{1/n}\ge B \text{ for all } (z,t)\in C^+_i \text{ and } |t| = \frac12~.\]
We have
 \[\sum_i (C^+_i\cdot (t=0))\ge (1-(m^2+1)\eps) d^n.\] 
Now the function $\tilde{\la}_n(z,t):=\frac1n \log |(df^n_t)(z(t))|$ is harmonic on $C^{+,*}_i$, 
satisfies $\tilde{\la}_n(z,t)\ge B>1$ over $|t|=1/2$, and is superharmonic on $C^+_i$
with a pole of order at least $\ge Ae^{\la(f_{\na})/2}$.
By the minimum principle, we get 
\[
\tilde{\la}_n(z,t) \ge \max\left\{\left(\frac12 \la(f_{\na})+ \frac1n \log A\right) \log|t|^{-1}, B\right\}
\]
which implies the theorem.
\end{proof}

\subsection{Existence of repelling cycles}
This section contains an auxiliary result used in the proof of the previous theorem.
We say that a real-analytic family of rational maps parame\-trized by the unit segment $[0,1]$
 is stable, if it can be extended to a stable holomorphic family on a neighborhood of $[0,1]$. 

\begin{thm}\label{thm:uniform-rep}
Let $g_\tau$ be a stable real-analytic family of rational maps of degree $d\ge2$ parametrized by the 
segment $\tau\in[0,1]$. 
Then for all $\eps>0$, there exists a constant $B>1$ such that 
for all $n\in\bN^*$, we have
\[
\frac1{d^n}
\# \left\{w\colon [0,1]\mathop{\rightarrow}\limits^{\cC^\omega} \bP^1_\bC, \, g^n_\tau(w(\tau))=w(\tau), \text{ and } \inf_{[0,1]} |dg^n_\tau(w(\tau))|^{1/n}\ge B\right\}
\ge 1-\eps~.
\]
\end{thm}
\begin{rem}
We think that the previous statement is true for any continuous family of rational maps without any assumption on stability. 
\end{rem}

\begin{proof}
Since the family is real-analytic, $g_t$ is the restriction of a holomorphic family defined on a small open neighborhood of $[0,1]$, 
which we may suppose to be simply connected. 
We may thus suppose that $g_t$ is in fact a holomorphic family parameterized by the unit disk $|t|<1$, and 
we are given a compact subset $K\subset \bD$ (corresponding to the image of the unit segment under the uniformization map). 
To simplify notation let $J_t$ be the Julia set of $g_t$. 

As $g_t$ is stable, there exists a holomorphic motion $\phi\colon \bD\times J_0\to \bP^1(\bC)$
such that $\phi(t,J_0)= J_t$ and
\[\phi(t,g_0(z))=g_t(\phi(t,z))\] 
for all $t$ and $z$. 
It follows that all periodic points in $J_0$ can be followed under the holomorphic motion. 
The number of non-repelling periodic points of $g_0$ is finite, say equal to $M$. For each integer $n$ we get 
$d^n-M$ distinct holomorphic functions $w_i(t)\in J_t$ such that $g_t^n(w_i(t))=w_i(t)$. 

Fix $\eps>0$. By Theorem~\ref{thm:main5} (see Remark~\ref{rem:also-complex}) there exists a constant $B>1$ such that 
$g_0$ has at least $(1-\eps)d^n$ periodic points $w_i$ whose multiplier
 satisfies $|dg^n_0(w_i)|\ge B^n$. 
 
 On a compact set  $K\subset \bD$, the maps $\phi(t,\cdot)$ are  Hölder with
 a uniformly bounded exponent $0<\kappa<1$, see~\cite[Corollary~4.4.8]{hubbard}. It follows that for any $t\in K$, we get 
$|dg^n_t(w_i(t))|\ge |dg^n_0(w_i)|^{\kappa n}$, hence the result. 
\end{proof}

\section{Existence of  period $2$ and $3$ rigid repelling cycles}

\subsection{Polynomial maps}
In this section, we prove the following result which was communicated to the author by V.~Huguin. 
The proof we give streamlines~\cite{huguin1}.
\begin{thm}\label{thm:period2}
Let $f$ be any polynomial of degree $d\ge2$ defined over $\Lau$
which has not potential good reduction. 

Then $f$ admits a rigid repelling periodic point of period $\le2$, 
and a repelling rigid fixed point if $d\in\{2,3\}$.
\end{thm}

The argument works for any tame polynomial map over any field in the sense of~\cite{trucco}. 

\begin{proof}
Suppose that $x$ is a rigid fixed point which is not repelling. 
The segment $[x,\infty]$ intersects the Julia set of $P$ at 
a point $\hat{x}$ which is a fixed type II point lying in the Julia set. 
We also have $\deg_{\hat{x}}(f)\ge2$, since otherwise the degree
function would be $1$ in a small neighborhood of $\hat{x}$ and
an open neighborhood of $\hat{x}$ in $[x,\infty]$ would be fixed.

Consider the closed ball $\bar{B}(\hat{x})\subset\bA^1$ whose boundary  is equal to $\hat{x}$. 
Since the residual characteristic of the ground field is $0$, the number of fixed point $F(\hat{x})$ (counted with mulitplicity) of $f$
inside $\bar{B}(\hat{x})$ is equal to $\deg_{\hat{x}}(f)$, and the number of critical points
(again counted with multiplicity) of $f$ in $\bar{B}(\hat{x})$ is equal to $\deg_{\hat{x}}(f) -1$.

\smallskip

Suppose that no rigid fixed point of $f$ is repelling. 
Let $y_1, \cdots , y_n$ be the set of type II fixed points of $f$
with local  degree $\ge2$, and let $N_1$ be the number of rigid critical
points lying in the union  of the balls $\bar{B}(y_i)$ (counted with mulitplicity). 
Since $\deg_{y_i}(f)\ge2$, we have 
$\deg_{y_i}(f) -1\ge \frac12\deg_{y_i}(f)$.
Also since $f$ has not potential good reduction, at least one rigid critical point
escapes to infinity, and thus $N_1\le d-2$. 
By what precedes, we get
\[
d-2 \ge N_1:= \sum_{i=1}^n (\deg_{y_i}(f)-1) 
\ge \frac12  \sum_{i=1}^n \deg_{y_i}(f)=\frac{d}2~.\]
When $d\le3$, we obtain a contradiction hence the existence of a rigid repelling fixed point. 
Note that we also get $n\ge2$.
\smallskip

Suppose now that $d\ge4$ and $f$ has no rigid repelling fixed and periodic $2$ point. 
To fix terminology, let us say that a $2$-cycle is a pair of distinct points
that is permuted by $f$.
The number of rigid $2$-cycles of a polynomial of degree $\delta$
is equal to $\frac12 (\delta^2-\delta)$. 

Consider as above the set $\{y_1,\cdots, y_n\}$ of all type II fixed points with local degree $\ge2$,
and write  $\delta_i=\deg_{y_i}(f)$.
Let  $\{\{z_1,f(z_1)\},\cdots, \{z_m,f(z_m)\}\}$ be the set of all type II $2$-cycles having local degree $\ge2$, and
write $\mu_i^+=\deg_{z_i}(f)$ and $\mu^-_i=\deg_{f(z_i)}(f)$. We shall suppose that 
$\mu^+_i\ge\mu^-_i$. 

The number of rigid fixed points in the union of the balls $\cup \bar{B}(y_i)$
is equal to 
\[\sum_{i=1}^n \delta_i= d\]
by our assumption that there is no repelling rigid fixed points. 
The number of $2$-cycles in the union of the balls $\cup \bar{B}(y_i)$ is equal to 
\[\sum_{i=1}^n \frac12 (\delta_i^2-\delta_i)~.\]
It follows that number of rigid $2$-cycles in the union of the balls $\cup \bar{B}(z_j) \cup \bar{B}(f(z_j))$
is equal to $\sum_{i=1}^m \mu_i^+ \mu_i^-$. 

Since no rigid $2$-cycle is repelling, we obtain 
\begin{equation}\label{eq:gez}
\sum_{i=1}^m \mu_i^+ \mu_i^-
= 
\frac12 (d^2-d)
- 
\sum_{i=1}^n \frac12 (\delta_i^2-\delta_i)
=
\sum_{i< j} \delta_i\delta_j
\ge 4{n\choose 2}= 2n(n-1)~.
\end{equation}
The number of critical points  lying in $\cup \bar{B}(y_i)$ is equal to 
\[\sum_{i=1}^n (\delta_i-1)= d-n~,\]
and there are 
\[
\sum_{i=1}^m (\mu_i^++\mu^-_i-2)
\]
critical points lying in $\cup \bar{B}(z_j) \cup \bar{B}(f(z_j))$.
It follows that 
\begin{equation}\label{eq:lez}
\sum_{i=1}^m (\mu_i^++\mu^-_i-2)\le (d-1) - (d-n) = n-1
~.
\end{equation}
We shall prove that~\eqref{eq:lez} implies 
$\sum_{i=1}^m \mu_i^+ \mu_i^-<2n(n-1)$ which yields a contradiction. 
Recall that by assumption $\mu_i^+\ge2$ and $\mu^+_i\ge\mu^-_i\ge1$.

When $n=2$, then $m=1$ and $\mu^+_1=2$, $\mu^-_1=1$, so that 
$\mu_1^+ \mu_1^-= 2<4=2n(n-1)$, as required. 

Suppose now that $n\ge3$. We argue by induction on $m$. 
When $m=1$, then 
\[
\mu_1^+ \mu_1^-
\le
\frac12(\mu_1^++\mu^-_1)^2
\le 
\frac12 (n+1)^2 < 2n(n-1)
~,\]
as was to be shown. 
Suppose that the result is proved for some $m\ge 1$, and 
pick $\mu_i^{\pm}$ with $i\in\{0,1, \cdots, m\}$
positive integers satisfying  $\mu_i^+\ge2$ and $\mu^+_i\ge\mu^-_i\ge1$, and~\eqref{eq:lez}. 
Set $\mu^+= \mu^+_0+\mu^-_1-1$ and $\mu^-=\mu^+_1+\mu^-_0-1$. 
Up to permuting indices we may suppose that $\mu^+\ge \mu^-\ge2$. 
Since 
\[
(\mu^++\mu^--2)
+
\sum_{i=2}^m (\mu_i^++\mu^-_i-2)
=
\sum_{i=0}^m (\mu_i^++\mu^-_i-2)
\le 
n-1
\] 
the induction step applies, and we get
\[
\sum_{i=0}^m \mu_i^+ \mu_i^-
= 
(\mu^+_0 \mu^-_0
+\mu^+_1\mu^-_1
- \mu^+\mu^-) +  \mu^+\mu^-
+ \sum_{i=2}^m \mu_i^+ \mu_i^-
< \Delta + 2n(n-1)
~,\]
where
\begin{align*}
\Delta &= 
\mu^+_0 \mu^-_0
+\mu^+_1\mu^-_1
- \mu^+\mu^-
\\
&= 
\mu^+_0 \mu^-_0
+\mu^+_1\mu^-_1
- (\mu^+_1+\mu^-_0-1)
(\mu^+_0+\mu^-_1-1)
\\
&=
- \mu^+_1(\mu^+_0-1)
- \mu^-_0(\mu^-_1-1)
+ \mu^+_0+\mu^-_1- 1
\end{align*}
When $\mu^-_1=1$, then 
$\Delta \le
- \mu^+_1(\mu^+_0-1) + \mu^+_0\le0$; and when $\mu^-_1\ge2$, then 
we get 
\begin{align*}
\Delta 
&=- \mu^+_1(\mu^+_0-1)
- \mu^-_0(\mu^-_1-1)
+ \mu^+_0+\mu^-_1- 1
\\&\le  - \mu^+_1(\mu^+_0-1) - 1 + \mu^+_0+\mu^-_1- 1
\le  (1-\mu^+_1)(\mu^+_0-1) -1+ \mu^+_1 \le0
\end{align*}
as well. 
This concludes the proof.
\end{proof}


\subsection{Existence of a period $3$ repelling cycle for cubic maps}
Recall that a cubic rational map over $\Lau$ which has not potential good reduction
has necessarily a positive Lyapunov exponent (since any affine Bernoulli map has degree $\ge4$, see~\S\ref{sec:lyap}).
And Theorem~\ref{thm:main5} shows that most rigid periodic cycles are repelling. 
In fact, we have:
\begin{thm}\label{thm:period3}
Suppose that $f$ is a rational map of degree $3$ defined over $\Lau$
which has not potential good reduction.

Then $f$ admits a repelling rigid periodic cycle of period at most $3$.
\end{thm}

As above, the argument works for any tame cubic rational map over any field.

\begin{proof}[Proof of Theorem~\ref{thm:period3}]
Let $x$ be a type II point. A direction at $x$ is a connected component of $\bP^{1,\an}_{\Lau}\setminus\{x\}$. 
The set of directions $T_x\bP^{1,\an}_{\Lau}$ at a type II point $x\in\bP^{1,\an}_{\Lau}$ is
in canonical bijection with $\bP^1(\bC)$. Given any $\vv\in T_x\bP^{1,\an}_{\Lau}$, we denote by $B(\vv)$
the open ball corresponding to $\vv$. Its boundary is $x$. Given any two type II points $x\neq x'$, we also 
let $A_{x,x'}$ be the unique connected component of $\bP^{1,\an}_{\Lau}\setminus\{x,x'\}$ whose boundary 
is equal to $\{x,x'\}$.

Any rational map $f$ mapping 
$x$ to $x'$ induces a map on the set of directions
\[Tf(x)\colon T_x\bP^{1,\an}_{\Lau}\to T_{x'}\bP^{1,\an}_{\Lau}\]
which is identified to a a rational map over $\bC$ of degree $=\deg_x(f)$, see~\cite[Theorem~7.34]{benedetto-book}. 

We shall make a frequent use of~\cite[Theorem~9.34]{baker-rumely}. If the local degree of $f$
remains constant on an open segment $(x,x')\subset\bP^{1,\an}_{\Lau}$,
then $f$ is injective on $[x,x']$.

Also, a result of Faber~\cite[Proposition~6.9]{faber1} states that the critical set
\[\cC_f=\left\{x\in \bP^{1,\an}_{\Lau},\ \deg_x(f)\ge2\right\}\]
is a tree whose endpoints are rigid critical points (and none of its component is reduced to a singleton).

\smallskip

Suppose that $f$ has not potential good reduction and does not contain any repelling
rigid fixed point.  By a theorem of Rivera-Letelier, see~\cite[Theorem~12.5]{benedetto-book} (or~\cite{rivera-periode})
then $f$ admits at least one fixed type II point $x$ for which $\deg_x(f)\ge2$. 
Denote by $x'$ the other preimage of $x$, and observe that since $\deg(f)=3$
the rational map $Tf(x')$ has degree $1$, and $Tf(x)$ has degree $2$.
Let $\vv$ be the tangent direction at $x$ pointing toward $x'$. 
Observe that the open annulus $A_{x,x'}$ between $x$ and $x'$ is mapped
properly by $f$ onto the ball $B(Tf(x)\cdot\vv)$. 
The degree of the map $f\colon A_{x,x'}\to B(Tf(x)\cdot\vv)$ is either $2$ or $3$; 
and the critical set $\cC_f$ of $f$ intersects the open segment $(x,x')$
as $f$ is not injective on $[x,x']$.

\smallskip

\paragraph*{Case 1a} Suppose first that $Tf(x)\cdot\vv\neq\vv$. 
Since the direction at $x'$ pointing toward $x$ is mapped to $Tf(x)\cdot\vv$, 
there exists a direction $\vv'\in Tx'$ which is mapped to $\vv$. 
It follows that $f\colon B(\vv')\to B(\vv)$ is an analytic isomorphism and since $ B(\vv')\Subset B(\vv)$
there exists a rigid repelling fixed point in $B(\vv')$ by Theorem~\cite[Theorem~12.5]{benedetto-book}.

\medskip

Suppose next that $Tf(x)\cdot\vv=\vv$. 

\paragraph*{Case 1b}
If $\vv$ is not critical  for $Tf(x)$, then it exists a (unique) $\ww\neq \vv$ 
which is not critical for $Tf(x)$ and mapped to $\vv$. Since $Tf(x')$ has degree $1$, 
there exists also a unique tangent direction $\ww'$ at $x'$ such that $Tf(x')\cdot\ww'=\ww$. 
This direction does not point towards $x$ because this one is mapped to $\vv$. 
Hence the open ball $B(\ww')$ is mapped one-to-one onto $B(\ww)$ 
which is mapped one-to-one also to $B(\vv)$. 
We get a map $f^2\colon B(\ww')\to B(\vv)$ such that  $B(\ww') \Subset B(\vv)$, 
hence a rigid
repelling period $2$ point in $B(\ww')$ as before. 

\medskip

Suppose now that  $\vv$ is critical for $Tf(x)$, and $\cC_f\cap [x,x']=[x,y]$ is connected.
Observe that $f$ is injective on both segments $[x,y]$ and $[y,x']$, 
hence $\deg_y(f)=3$. It follows that $\cC_f$ is connected, and can only branch at $y$.  
The tangent vector $\vv_0$ at $y$
pointing towards $x'$ is mapped with degree $1$ to the
tangent vector $\vv_1$ at $f(y)$ pointing to $x$, and  $f\colon B(\vv_0)\to B(\vv_1)$
is an analytic isomorphism.

\medskip

\paragraph*{Case 1c}
When $y$ does not belong to the interior of the segment $[x,f(y)]$, then 
$B(\vv_0)\Subset B(\vv_1)$ hence $f$ contains a rigid repelling fixed point in $B(\vv_0)$.

\medskip

Otherwise, we have $[x,f(y)] \cap [x,x']= [x,y_*]$ with $x<y<y_*<x'$. 

\medskip

\paragraph*{Case 1d}
Suppose $y_*\neq f(y)$. There exists 
a point $y'_*\in(x,y)$ such that $f(y'_*)=y_*$. 
Observe that $\deg(y'_*)=2$, and  the directions
$\vv(x), \vv(x')$ at $y'_*$ pointing towards $x$ and $x'$
respectively are both critical.
Let $\vv_*$ be the direction at $y_*$  pointing to $x$. 
Then $\vv(x)$ is mapped to the direction at $y_*$ pointing to $x$ with degree $2$, and
$\vv(x')$ is mapped to the direction at $y_*$ pointing to $f(y)$. The latter direction
 is not $\vv_*$ since by assumption is not $f(y)\neq y_*$. 
We conclude that there exists a direction $\vv'_*$ at $y'_*$
which is mapped one-to-one to $\vv_*$. 
The image of $y_*$ lies in the segment $[f(y),x]$ at distance $d_\bH(y,y_*)$
of $f(y)$, hence $f(y_*)\in [y,f(y)]$. Moreover the direction $Tf(f(y_*))\cdot \vv_*$ points
towards $x$, hence $B(Tf(f(y_*))\cdot \vv_*)$ contains $B(\vv'_*)$. 
Observing that $f^2\colon B(\vv'_*)\to B(Tf(f(y_*))\cdot \vv_*)$, we get 
 a rigid repelling period $2$ point in $B(\vv'_*)$.

 \medskip

\paragraph*{Case 1e}
It remains to treat the case $f(y)\in[x,x']$.
Observe that in this case $f^2(y)=y$
because $f$ doubles the length on $[x,y]$ and preserves it on $[y,x']$.


Let $y_1$ be the preimage of $y$ in the segment $(x,y)$. 
Since $f$ maps injectively the segment $[f(y),x']$ onto $[y,x]$, 
$y_1$ admits a preimage $y_2\in [f(y),x']$. 
Finally both directions at $y$ pointing to $x$ and $x'$ are mapped to the direction at $f(y)$ pointing to $x$, hence
there exists a direction $\ww$ at $y$ mapped to the direction at $f(y)$ pointing to $x'$. 
In particular, we may find a preimage $y_3$ of $y_2$ in the ball $B(\ww)$ such that $Tf(y_3)$ 
maps the direction pointing to $y$ to the one pointing to $f(y)$.

Observe that $\deg(Tf(y_1))=2$
and the critical directions are the one pointing to $x$ and $x'$. 
It follows that there are two non critical directions $\{\ww_1,\ww'_1\}$ 
at $y_1$ mapped to $\ww$. 
Take any direction $\ww_2$ (resp. $\ww'_2$) at $y_2$
mapped to $\ww_1$ (resp. $\ww'_1$). 
These directions do not point to neither $x$ nor $x'$. 
We claim that either $\ww_2$ or $\ww'_2$ admits a non-critical
preimage at $y_3$.
Grant this claim, and suppose $\ww_3$ is a non-critical direction at $y_3$
mapped to $\ww_1$. 
Then we have a sequence of analytic isomorphisms
\[
B(\ww_3)
\mathop{\rightarrow}\limits^f 
B(\ww_2)
\mathop{\rightarrow}\limits^f 
B(\ww_1)
\mathop{\rightarrow}\limits^f 
B(\ww) \Supset B(\ww_3)
\]
so that we get a rigid period $3$ repelling point in $B(\ww_3)$.

We prove the claim by contradiction. 
If $\deg(Tf(y_3))=2$, then the preimages of 
$\ww_2$ and $\ww'_2$ would be the only critical directions. 
However the direction pointing to $y$ needs to be critical since we assume
that $\cC_f$ is connected. 
The same argument applies when $\deg(Tf(y_3))=3$.

\medskip
 
In the remaining of the proof we suppose   that  $\vv$ is critical for $Tf(x)$, and $\cC_f\cap [x,x']$ is disconnected.
Then no point can have local degree $3$. It follows that
$\cC_f\cap [x,x']=[x,y] \cup \{z\}$, and both maps 
$f|_{[x,z]}$ and $f|_{[z,x']}$ are injective. Observe that the distance is doubled on $[x,y]$
hence it is not possible that $f(z)=z$.

\medskip

\paragraph*{Case 2a}
If $z$ is not in the interior of $[x,f(z)]$, then the tangent direction $\vv'$ at $z$ pointing to $x'$ is mapped
to the direction $\vv$ at $f(z)$ pointing to $x$. The map
$f\colon B(\vv')\to B(\vv)$
is an analytic isomorphism, and  $f$ contains a rigid repelling fixed point in $B(\vv')$. 

\medskip

Suppose that $z$ belongs to the interior of $[x,f(z)]$.
Let $l=d_\bH(x,y)$ and $L=d_\bH(z,x')$. 
Observe that $d_\bH(x,f(z))=L$, and 
\[
L=d_\bH(x,f(z))=2 d_\bH(x,y)+d_\bH(y,z)\]
so that $d_\bH(y,z)=L-2l >0$, and $d_\bH(z,f(z))=l$.
Let $y_*=[z,x']\cap[f(z),x']$, and $\delta=d_\bH(z,y_*)$.
Since $d_\bH(z,f(z))< d_\bH(z,x')$, we have
$y_*\in (z,x')$.
There a preimage $y'_*$ of $y_*$ in the segment
$(x,z)$. 

\medskip

\paragraph*{Case 2b}
Suppose the direction $\vv_*$ at $y_*$
pointing to $x'$ admits a preimage $\vv'_*$ at $y'_*$
which does not point to $x'$.
Since  $d_\bH(f(y_*),f(z))=d_\bH(y_*,z)=\delta\le l$
the point $f(y_*)$ belongs to the segment $[z,f(z)]$, 
and $f^2$ is an analytic isomorphism
on $B(\vv'_*)$ for which
$B(\vv'_*)\Subset f^2B(\vv'_*)$
which yields a rigid period $2$ repelling point.

\medskip

Otherwise, $y'_*=y$ and the only preimage of $\vv_*$
is the critical direction $\vv_c$ at $y$ not pointing to $x$. 
We are thus in the situation where
$x<y<z<f(y)<x'$ all belong the same segment, and 
$[f(y),f(z)]$ is a branch at $f(y)$ different from the ones pointing
to $x$ and $x'$. 
The point $f^2(y)$ belongs to the segment $[f(z),x]$. 

\medskip

\paragraph*{Case 2c}
If  $f^2(y)$ lies in the interior of $[f(y),f(z)]$, then the direction $\vv_*$ at
$f(y)$ pointing to $x'$ is mapped to the direction $\vv_+$ at $f^2(y)$
pointing to $z$ (or $x$ or $x'$ which are all the same).
We get a repelling rigid fixed point in the ball $B(\vv_*)$.

\medskip

\paragraph*{Case 2d}
We can assume that $f^2(y)\in (x,f(y)]$.
The ball $B(\vv_c)$ is mapped as a $2$-to-$1$ cover onto
$B(\vv_*)$, and $B(\vv_*)$ is mapped to the open ball $B$ whose boundary
is $f^2(y)$ and contains $x$. 
Since $d_\bH(f^2(y),f(z))=3l-L<l$, $f^2(y)\in[z,f(z)]$ and $B(\vv_c)\Subset B$.

We have two sub-cases.
Either there exists a rigid repelling period $2$ point in $B(\vv_c)$; 
or there is a type II period $2$ point $w$ in $B(\vv_c)$ having local degree $\ge2$. 
In the latter case, consider the  preimage $y_1$ of $y$ in the segment $[x,y]$. 
Pick one non-critical direction $\vv_d$ at $y_1$ mapped by $Tf(y_1)$
to $\vv_c$. Then $f^3$ maps again $2$-to-$1$ the ball
$B(\vv_d)$ to $B$. If there is no period $3$ rigid repelling point in $B(\vv_d)$, 
then we get a period $3$ type II point $w_1$ having local degree $\ge2$. 
Its image lies in the ball $B(\vv_c)$ and is critical. 
But the critical set of $f$ inside $B(\vv_c)$ is a ray ending at $y$, hence
either $f(w_1)\in (w,y]$ or $w\in (f(w_1),y]$. 
The former case is impossible because otherwise $w$ would be in the Fatou set, 
and the latter case is also impossible because then $f(w_1)$
(hence $w_1$) lies in the Fatou set. 
 
The proof is complete. 
\end{proof}
\begin{figure}[ht]
    \centering
    \includegraphics[width=0.8\textwidth]{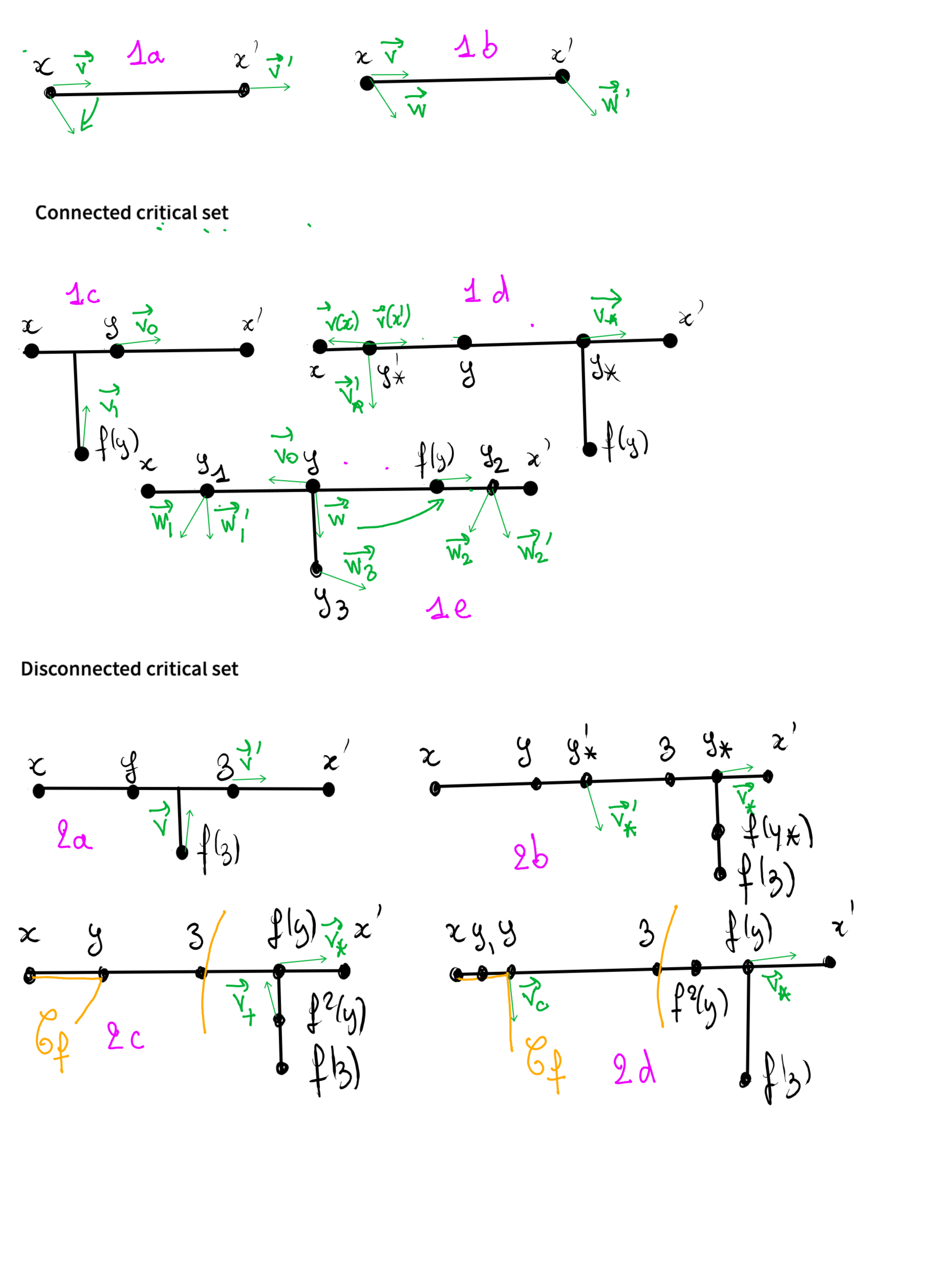}
    \caption{Proof of Theorem~\ref{thm:period3}}
    \label{fig:tikzpgf}
\end{figure}

\subsection{Proof of Theorem~\ref{thm:main2} and Corollary~\ref{cor:embed}}

Let $f_t$ by any meromorphic family parameterized by $t\in\bD$. 
Saying that $f_t$ has no fixed point for which the multiplier is blowing-up at $0$
is equivalent to say that the multiplier function $(z,t)\mapsto df_t(z)$ is holomorphic
on the curve $\{f_t(z)=z,\ |t|\le1/2\}$. 
By Lemma~\ref{lem:multi-na}, this is equivalent to say that no rigid fixed points of $f_{\na}$
is repelling. 

It is proved in~\cite{milnor-quadra} that the map sending a quadratic rational map 
$f$ to the symmetric functions of the multipliers at its three fixed points, 
is a proper biholomorphism onto its image (which turns out to be a plane). 
Statement (1) of Theorem~\ref{thm:main2} thus follows in that case. 

Suppose that $f_t$ is a family of polynomials of degree $d\ge4$, and $f_{\na}$ has no repelling rigid periodic points of period $\le2$ (resp. $1$). Then by Theorem~\ref{thm:period2} 
we conclude that $f_{\na}$ has potential good reduction, hence the family is not degenerating by Theorem~\ref{thm:pot-good}. Statements (2) (resp. (1) for cubic polynomials) follow.
The same argument apply for cubic rational maps with  Theorem~\ref{thm:period3} in place of  Theorem~\ref{thm:period2}. 
The proof of Theorem~\ref{thm:main2} is complete.

\bigskip

Let us prove Corollary~\ref{cor:embed} (1) (the first statement is completely analogous).
Consider the multiplier function $\Lambda\colon \rat_3\to \bA^4\times\bA^3\times \bA^8$ sending the conjugacy class of a 
cubic rational map to the symmetric polynomials of the multipliers at fixed points, period $2$ and period $3$ cycles.
Then $\Lambda$ is a regular algebraic map on $\rat_3$ hence extends to a rational map
$\Lambda\colon\overline{\rat}_3\dasharrow  \bP^4\times\bP^3\times \bP^8$. Fix any proper birational morphism 
$\pi\colon\widehat{\rat}_3\to\overline{\rat}_3$
such that $\pi$ is an isomorphism over $\rat_3$, and $\Lambda$ is a regular function on $\widehat{\rat}_3$. 

Since  $\widehat{\rat}_3$ is projective, then $\Lambda$ is proper, and $Z:=\Lambda(\pa \widehat{\rat}_3)$ is an algebraic
subvariety of $\bP^4\times\bP^3\times \bP^8$. When $Z$ does not intersect $ \bA^4\times\bA^3\times \bA^8$, then 
the valuative criterion of properness applies and proves that $\Lambda\colon \rat_3\to \bA^4\times\bA^3\times \bA^8$ is proper. 
Otherwise, there is a point $p\in \pa\widehat{\rat}_3$ whose image lies in $ \bA^4\times\bA^3\times \bA^8$. 
Choose any curve $C$ containing $p$ but not contained in $\pa\widehat{\rat}_3$. 
Pick any local holomorphic  parametrization $\gamma\colon \bD \to C$ mapping $0$ to $p$. Then choosing a sufficiently large integer $q\in\bN^*$,  one can find a family of rational maps $f_t$ such that its class of conjugacy is equal to $\gamma(t^q)$. 
Then the multiplier function on periodic points of period $\le 3$ is bounded on the family hence 
$f_t$ has potential good reduction by Theorem~\ref{thm:main2}, which is absurd.


\end{document}